\documentclass[a4paper,12pt]{amsart}
\usepackage{amsmath}
\usepackage{amssymb}
\usepackage{mathrsfs}
\usepackage{enumerate}
\usepackage{ifthen}
\usepackage{graphicx}
\usepackage{caption}
\usepackage[T1]{fontenc} 
\usepackage{tikz}
\usepackage{cite}

\usepackage{lineno}
\nonstopmode \numberwithin{equation}{section}
\usepackage{ifpdf}
\ifpdf
  \usepackage[hyperindex,pagebackref]{hyperref}%
\else
  \expandafter\ifx\csname dvipdfm\endcsname\relax
    \usepackage[hypertex,hyperindex,pagebackref]{hyperref}
  \else
    \usepackage[dvipdfm,hyperindex,pagebackref]{hyperref}
  \fi
\fi

\newtheorem{theorem}{Theorem}[section]

\newtheorem{corollary}{Corollary}[section]

\newtheorem{problem}{Problem}[section]

\theoremstyle{remark}
\theoremstyle{definition}
\newtheorem{remark}{Remark}[section]

\newtheorem{example}{Example}[section]

\theoremstyle{plain}

\newtheorem*{lemA}{Lemma A}

\numberwithin{equation}{section}
\numberwithin{theorem}{section}

\newtheorem{lem}{Lemma}

\newcounter{minutes}
\divide\time by 60
\newcounter{hours}
\multiply\time by 60
\addtocounter{minutes}{-\time}

\makeatletter
\@namedef{subjclassname@2020}{\textup{2020} Mathematics Subject Classification}
\makeatother

\begin{document}

\title{Pre-Schwarzian and Schwarzian norm Estimates for Robertson class}

\author{Molla Basir Ahamed}
\address{Molla Basir Ahamed, Department of Mathematics, Jadavpur University, Kolkata-700032, West Bengal, India.}
\email{mbahamed.math@jadavpuruniversity.in}

\author{Rajesh Hossain}
\address{Rajesh Hossain, Department of Mathematics, Jadavpur University, Kolkata-700032, West Bengal, India.}
\email{rajesh1998hossain@gmail.com}

\author{Xiaoyuan Wang}
\address{School of Mathematical Sciences, Liaocheng University, Liaocheng
	252000, Shandong, People's Republic of China}
\email{mewangxiaoyuan@163.com}


\def\thefootnote{}
\footnotetext{ {\tiny File:~\jobname.tex,
printed: \number\year-\number\month-\number\day,
          \thehours.\ifnum\theminutes<10{0}\fi\theminutes }
} \makeatletter\def\thefootnote{\@arabic\c@footnote}\makeatother

\begin{abstract}
Let $\mathcal{A}$ denote the class of analytic functions $f$ on the unit disk $\mathbb{D}=\{z\in\mathbb{C} : |z|<1\}$, normalized by $f(0)=0$ and $f^{\prime}(0)=1$. For $-\pi/2<\alpha<\pi/2$, let $\mathcal{S}_{\alpha}$ be the subclass of $\mathcal{A}$ consisting of functions $f$ that satisfy the relation $\mathrm{Re}\{e^{i\alpha}\left(1+zf^{\prime\prime}(z)/f^{\prime}(z)\right)\}>0$ for $z\in\mathbb{D}$. In this paper, we first give an equivalent characterization for a subclass of Robertson functions; then we present the distortion and growth theorems and obtain the pre-Schwarzian and Schwarzian norms for the subclass $\mathcal{S}_{\alpha}$. In addition, a sharp upper bound of the Schwarzian norm for the subclass is given in terms of the value $f^{\prime \prime}(0)$.
\end{abstract}
\maketitle
\pagestyle{myheadings}
\markboth{M. B. Ahamed, R. Hossain, and X. Y. Wang}{Pre-Schwarzian and Schwarzian norm Estimates for Robertson class}

\section{\textbf{Introduction}}
The pre-Schwarzian and Schwarzian derivatives play an important role in geometric function theory, especially using pre-Schwarzian and Schwarzian derivatives embedding models to characterize Teichmm\"{u}ller spaces. They are also important tools for studying the inner radius of univalency for planer domains and quasiconformal extensions, see \cite{Lehto-1987,Lehto-JAM-1979}.  The research of pre-Schwarzian and Schwarzian derivatives has a long history and is closely related to other disciplines. The concept of Schwarzian derivative originated from the works of Kummer as early as $1836$.
Pre-Schwarzian and Schwarzian derivatives  are closely related to conformal mappings and have found significant applications in complex analysis. Pre-Schwarzian derivatives, as a precursor or generalization, also play a crucial role in understanding the properties of functions and their transformations.
In the broader context of function theory, Schwarzian derivatives are essential for understanding the properties of analytic and meromorphic functions. Several sufficient conditions for univalent analytic functions are obtained by using the notions of pre-Schwarzian and Schwarzian derivatives.
\subsection{Pre-Schwarzian and Schwarzian derivatives of analytic functions.}
For a locally univalent analytic function $f(z)$ defined in a simply connected domain ${\Omega}$, the pre-Schwarzian derivative $Pf$ and the Schwarzian derivative $Sf$ are defined by
\begin{align} Pf=\frac{f^{\prime\prime}(z)}{f^{\prime}(z)}\;\mbox{and}\;Sf=(Pf)^{\prime}(z)-\frac{1}{2}(Pf)^2(z)=\frac{f^{\prime\prime\prime}(z)}{f^{\prime}(z)}-\frac{3}{2}\left(\frac{f^{\prime\prime}(z)}{f^{\prime}(z)}\right)^2
\end{align}
respectively. The pre-Schwarzian and Schwarzian norms are
\begin{align}
	||Pf||_{\Omega}= \sup_{z\in{\Omega}}|Pf|\eta^{-1}_{\Omega}\;\mbox{and}\;||Sf||_{\Omega}= \sup_{z\in{\Omega}}|Sf|\eta^{-2}_{\Omega}
	\end{align}
	respectively, where $\eta_{\Omega}$ is the Poincare density. In particular, when ${\Omega}=\mathbb{D}$, denote $||Sf||_{\Omega}$ and $||Pf||_{\Omega}$ by $||Sh||$ and $||Ph||$, respectively.\vspace{2mm}

Let $\mathcal{A}$ denote the subclass of $\mathcal{H}$ consisting of functions $f$ with normalized conditions $f(0)=f^{\prime}(0)-1=0$. Thus, any function $f$ in $\mathcal{A}$ has the Taylor series expansion of the form
\begin{align}\label{Eq-1.1}
	f(z)=z+\sum_{n=2}^{\infty}a_nz^n\;\; \mbox{for all}\;\; z\in\mathbb{D}.
\end{align}
Let $\mathcal{S}$ be the subclass of $\mathcal{A}$ consisting of univalent (that is, one-to-one) functions. A function $f\in \mathcal{A}$ is called starlike (with respect to the origin) if $f(\mathbb{D})$ is starlike with respect to the origin, and convex if $f(\mathbb{D})$ is convex. The class of starlike (resp. of convex) functions are denoted by $\mathcal{S}^*$ (resp. $\mathcal{C}$).\vspace{2mm}

The necessary conditions for univalent analytic function $f$ defined in $\mathbb{D}$, it is easy to know that the pre-Schwarzian norm $||Pf||\leq6$, see   \cite[p. 274]{Duren-1983}.
The sufficient condition  for univalence   obtained by  Becker \cite{Becker-JRAM-1972,Becker-Prommerenke-JRAM-1972}
by using the pre-Schwarzian derivative in $1972$, the conclusion is that if $||Pf||\leq1$, then the function $f$ is univalent in $\mathbb{D}$.
In 1976, Yamashita \cite{Yamashita-MM-1976} gave an equivalent characterization
of $\|Pf\|$ is finite and $f$ is uniformly locally univalent in $\mathbb{D}$.
Sugawa \cite{Sugawa-AUMCDS-1996} studied pre-Schwarzian norm for the strongly starlike functions of order $\alpha\ (0<\alpha\leq1)$. Yamashita \cite{Yamashita-HMJ-1999} generalized Sugawa's results by
obtained pre-Schwarzian norm for some general class named  Gelfer classes, these classed  contain he classical starlike, convex, close-to-convex, also, included Sugawa's result.
The subclass of $\alpha$-spirallike functions (-$\pi/2<\alpha<\pi$/2) was studied by
Okuyama \cite{Okuyama-CVTA-2000}  and later  extended to a general class called $\alpha$spirallike functions of order $\rho$ $(0\leq\rho<1)$  by Aghalary and Orouji \cite{Aghalary-Orouji-COAT-2014}. The sharp estimate of the pre-Schwarzsian norm for Janowski starlike functions was studied  Ali and Pal \cite{Ali-Pal-MM-2023}. Recently,  Wang et al \cite{Wang-Li-Fan-MM-2024} considered the sharp bounds of pre-Schwarzian norm for Janowski convex functions.
Other subclasses have also been widely, such as meromorphic function exterior of the unit disk \cite{Ponnusamy-Sugawa-JKMS-2008}, subclass of strong starlike function \cite{Ponnusamy-Sahoo-M-2008}, uniformly convex and uniformly starlike function \cite{Kanas-AMC-2009} and bi-univalent function \cite{Rahmatan-Najafzadeh-Ebadian-BIMS-2017}, other subclasses, see \cite{Kanas-Maharana-Prajapat-JMAA-2019,Liu-Ponnusamy-IM-2018}. For the pre-Schwarzian norm estimates of other  forms such as the convolution operator and the integral operator, see \cite{Choi-Kim-Ponnusamy-Sugawa-JMAP-2005,Kim-Sugawa-PEMS-2006,Parvatham-Ponnusamy-Sahoo-HMJ-2008,Ponnusamy-Sahoo-JMAA-2008}. \vspace{2mm}

Like pre-Schwarzian norm, Schwarzian norm has also been widely studied.
The pioneering work of the necessary conditions for a univalent function $f\in\mathcal{A}$ such that $\|Sf\|\leq6$ was first discovered by Kraus \cite{Kraus-1932} and later rediscovered by Nehari \cite{Nehari-BAMS-1949}. In the same paper, Nehari  also proved that if $\|Sf\|\leq2$, then the function $f$ is univalent in $\mathbb{D}$, also see \cite{Hille-BAMS-1949}. Ahlfors and Weill \cite{Ahlfrors-Weill-PAMS-2012} proved that if Schwarzian norm $\|Sf\|\leq 2k$, where $0\leq k<1$, then $f$ can be quasiconformal extended to the entire complex plan $\bar{\mathbb{C}}$. On the contrary, K\"uhnau \cite{Kuhnau-MN-1971} showed that if $f$ is a $k$-quasiconformal  mapping on $\bar{\mathbb{C}}$ then $\|Sf\|\leq6k$. Regarding the estimates of the Schwarzian norm for the subclasses of univalent functions, Fait et al.\cite{Fait-Krzyz-CMH-1976} studied the strong starlike functions. In 1996, Suita \cite{Suita-JHUEDS-1996} studied the sharp estimate of $\mathcal{C(\alpha)}$. Okuyama \cite{Okuyama-CVTA-2000} studied 	$\alpha$-spirallike functions. Kanas and Sugawa \cite{Kanas-Sugawa-APM-2011} studied the subclasses of strongly starlike functions of order $\alpha$ $(0 < \alpha\leq1)$ and uniformly convex function. Recently,	Schwarzian norm estimates for other subclasses of univalent functions have been gradually studied by many people, such as class of concave functions \cite{Bhowmik-Wriths-CM-2012}, Robertson class \cite{Ali-Pal-BDS-2023} and other univalent analytic subclasses, see \cite{Ali-Pal-PEMS-2024}. Therefore, using the pre-Schwarzian
and Schwarzian norms to study the univalence and quasiconformal extension problems of analytic function arouses a new wave of research interest. Carrasco et al.
\cite{Carrasco-Hernandez-AMP-2023} considered convex mappings of order  $\alpha$ and Wang et al. \cite{Wang-Li-Fan-MM-2024} generalized it to the class of Janowski convex functions. Following this approach, this article will consider
some properties of Robertson class that has not been considered by previous researchers.


\section{\textbf{Pre-Schwarzian and Schwarzian norm Estimates for Robertson class}}
A domain $\Omega$ containing the origin is called $\alpha$-spirallike if for each point $\omega_0$ in $\Omega$ the arc of the $\alpha$-spiral from the origin to the point $\omega_0$ entirely lies in $\Omega$. A function $f\in\mathcal{A}$ is said to be an $\alpha$-spirallike if
\begin{align*}
	{\rm Re}\left(e^{i\alpha}\frac{z f^{\prime}(z)}{f(z)}\right)>0\;\mbox{for}\; z\in\mathbb{D},
\end{align*}
where $|\alpha|<\pi/2$. In $1933$, $\check{S}$pa$\check{c}$ek (see \cite{Spacek-1933}) introduces and studied the class of $\alpha$-spirallike functions and this class is denoted by $\mathcal{SP}(\alpha)$. Later on, Robertson \cite{Robertson-MMJ-1969} introduced a new class of functions, denoted by $\mathcal{S}_\alpha$, in connection with $\alpha$-spirallike functions. A function $f\in\mathcal{A}$ is in the class $\mathcal{S}_\alpha$ if, and only if,
\begin{align*}
	{\rm Re}\left(e^{i\alpha}\left(1+\frac{z f^{\prime\prime}(z)}{f^{\prime}(z)}\right)\right)>0\;\mbox{for}\; z\in\mathbb{D}.
\end{align*}
We note that $f\in\mathcal{A}$ is in the class $\mathcal{S}_\alpha$ if and only if $zf^{\prime}(z)\in\mathcal{SP}(\alpha)$.\vspace{2mm}

Before we state our main result, let us recall another important and useful tool known as the differential subordination technique. Many problems in geometric function theory can be solved in a simple and sharp manner with the help of differential subordination.  A function $f\in\mathcal{A}$ is said to be subordinate to another function $g\in\mathcal{A}$ if there exists an analytic function $\omega : \mathbb{D}\to\mathbb{D}$ with $\omega(0)=0$ such that $f(z)=g(\omega(z))$ and it is denoted by $f\prec g$. Moreover, when $g$ is univalent, then $f\prec g$ if, and only if, $f(0)=g(0)$ and $f(\mathbb{D}\subset g(\mathbb{D})$. \vspace{2mm}

In terms of the subordination, for a function $f\in\mathcal{A}$, we have
\begin{align}\label{Eq-2.1}
	f\in\mathcal{S}_\alpha\Longleftrightarrow e^{i\alpha}\left(1+\frac{z f^{\prime\prime}(z)}{f^{\prime}(z)}\right)\prec \frac{e^{i\alpha}+e^{-i\alpha}z}{1-z}.
\end{align}
In particular, when $\alpha=0$, the class $\mathcal{S}_\alpha$ reduces to the classical class $\mathcal{C}$ of convex functions. For general values of $\alpha$, Roberson \cite{Robertson-MMJ-1969} proved that functions in the class $\mathcal{S}_\alpha$ need not be univalent in $\mathbb{D}$. By using Nehari's test, Robertson \cite{Robertson-MMJ-1969} also proved that functions in the class $\mathcal{S}_\alpha$ are univalent in $\mathbb{D}$ if $\alpha$ satisfies the inequality $0<\cos\alpha\leq x_0\approx 0.2034...$, where $x_0$ is the positive root of the equation
\begin{align*}
	16x^3+16x^2+x-1=0.
\end{align*}
Later, Libera and Zeigler \cite{Libera-Ziegler-TAMS-1972} improved the range of $\alpha$ to $0 < \cos \alpha \leq x_0 \approx 0.2564...$ in 1972. In 1975, Chichra \cite{Chichra-PAMS-1975} further improved the range to $0 < \cos \alpha \leq x_0 \approx 0.2588...$. Interestingly, in the same year, Pfaltzgraff \cite{Pfaltzgraff-BLMS-1975} proved that the functions in the class $\mathcal{S}_\alpha$ are univalent if $0 < \cos \alpha \leq x_0 \approx 1/2$. Currently, this remains the best improvement for $\alpha$ for which functions in $\mathcal{S}_\alpha$ are univalent in $\mathbb{D}$. In 1977, Sing and Chichra \cite{Singh-Chichra-IJPAM-1977} proved that if $f \in \mathcal{S}_\alpha$ with $f^{\prime\prime}(0) = 0$, then the function $f$ is univalent in $\mathbb{D}$ for all values of $\alpha$, where $|\alpha| < \pi/2$. \vspace{2mm}

    In \cite{Chuaqui-Duren-Osgood}, Chuaqui \emph{et. al.} proved a result by applying the Schwarz-Pick lemma and the fact that the expression $1+z(f^{\prime\prime}/f^{\prime})(z)$ is subordinate to the half-plan mapping $\ell(z)=(1+z)/(1-z)$, which is
    \begin{align}\label{Eq-2.2}
        1+\frac{zf^{\prime\prime}(z)}{f^{\prime}(z)}=\ell(w(z))=\frac{1+w(z)}{1-w(z)}
    \end{align}
for some function $w : \mathbb{D}\to\mathbb{D}$ holomorphic and such that $w(0)=0$. \vspace{2mm}

The expression which is defined in \eqref{Eq-2.2} allowed us to obtain other characterizations for the convex functions:
\begin{align}\label{Eq-2.3}
    f\in\mathcal{C}\; \mbox{if, and only if,}\; {\rm Re}\left(1+\frac{zf^{\prime\prime}(z)}{f^{\prime}(z)} \right)\geq \frac{1}{4}\left( 1-|z|^2\right)\bigg| \frac{f^{\prime\prime}(z)}{f^{\prime}(z)}  \bigg|^2,
\end{align}
and
\begin{align}\label{Eq-2.4}
    f\in\mathcal{C}\; \mbox{if, and only if,}\; \bigg|\left(1-|z|^2\right)\frac{f^{\prime\prime}(z)}{f^{\prime}(z)} -2\bar{z} \bigg|\leq 2,
\end{align}
for all $z\in\mathbb{D}$.\vspace{2mm}

The Schur class $\mathbb{S}$ is defined by
\begin{align*}
    \mathbb{S}=\{f : \mathbb{D}\to \overline{\mathbb{D}} : f\; \mbox{is analytic}\}.
\end{align*}
By the Maximum Modulus Principle, we know that $f(z)\in\mathbb{T}:=\partial \mathbb{D}$ for some $z\in\mathbb{D}$ if, and only if, $f$ is a constant function with $|f(z)|=1$. Thus, the class $\mathbb{S}$ consists of non-constant analytic functions $f: \mathbb{D}\to\mathbb{D}$ along with the constant functions with values in $\overline{\mathbb{D}}$. A fundamental cornerstone of complex analysis is the Schwarz lemma. Its deceptive simplicity belies its many profound consequences \cite{Dineen-1989}. Although Schwarz first proved this lemma for injective functions, Carath\'eodory later provided the proof for its general version.
\begin{lemA}(Schwarz Lemma)\cite{Schwarz-1890}
    If $f\in\mathbb{S}$ with $f(0)=0$, then
    \begin{enumerate}
        \item[\emph{(i)}] $|f(z)|\leq |z|$ for all $z\in\mathbb{D}$, and
        \item[\emph{(ii)}] $|f^{\prime}(0)|\leq 1$.
    \end{enumerate}
    Moreover, $|f(w)|=|w|$ for some $w\in\mathbb{D}\setminus \{0\}$, or if $|f^{\prime}(0)|=1$, then there is a $\zeta\in\mathbb{T}$ so that $f(z)=\zeta z$ for all $z\in\mathbb{D}$.
\end{lemA}
Let $\mathbb{B}$ denote the set of all functions $\omega$ that are analytic in $\mathbb{D}$ and satisfy $|\omega(z)|\leq 1$ for all $z\in\mathbb{D}$. We also consider the subfamily $\mathbb{B}_0:=\{\omega\in\mathbb{B} : \omega(0)=0\}$.\vspace{2mm}

In this section, using the Schwarz lemma, we firstly give the equivalent characterization of \eqref{Eq-2.3} and \eqref{Eq-2.4} for the class $\mathcal{S_\alpha}$ (Robertson class) and obtain Theorem \ref{Th-4.1}. Next, we present the distortion and growth theorem (Theorem \ref{Th-4.2}) and derive the pre-Schwarzian and Schwarzian norms for $\mathcal{S_\alpha}$ in terms of $f''(0)$. Our results generalize to the classical class $\mathcal{C}$ of convex functions, as shown through corollaries. Proofs of the theorems follow the main results.
\begin{theorem}\label{Th-4.1} For $-\pi/2<\alpha<\pi/2$, the following are equivalent:
	\begin{enumerate}
		\item[\emph{(i)}] $f\in\mathcal{S_\alpha}$;\vspace{2mm}
		
		\item[\emph{(ii)}]
			\begin{align}\label{Eq-2.5}
				{\rm Re}\left(1+e^{i\alpha}\frac{zf^{\prime\prime}(z)}{f^{\prime}(z)}\right)\geq1- \cos\alpha+\left(\frac{1-|z|^2}{4\cos\alpha}\right)\bigg|\frac{f^{\prime\prime}(z)}{f^{\prime}(z)}\bigg|^2;
			\end{align}
            \noindent Sharp inequality holds for all $z\in\mathbb{D}$ unless
            \begin{align*}
                f^{\prime}(z)=\frac{1}{\left(1-z\zeta \right)^{2e^{-i\alpha}\cos\alpha}}
            \end{align*}
            for some $\zeta\in\mathbb{T}$.
		 \vspace*{2mm}
		
		\item[\emph{(iii)}] $\displaystyle	\bigg|	(1-|z|^2)\left(\frac{f^{\prime\prime}(z)}{f^{\prime}(z)}\right)-(2\cos\alpha)\bar{z}\bigg|\leq2\cos\alpha.
		$
	\end{enumerate}
\end{theorem}
\begin{proof}[\bf Proof of Theorem \ref{Th-4.1}]
First, we will prove that $(i)$ is equivalent to $(ii).$ For $-\pi/2<\alpha<\pi/2$, let  $h\in\mathcal{S_\alpha}$. Then, by assumption, it follows from \eqref{Eq-2.1} that
\begin{align}\label{Eq-4.1}
	e^{i\alpha}\left(1+\frac{zf^{\prime\prime}(z)} {f^{\prime}(z)}\right)\prec\frac{e^{i\alpha}+e^{-i\alpha}z}{1-z}.
\end{align}
Consequently, there exists an analytic function $\omega\in\mathbb{B}_0$, \emph{i.e.,}  $\omega : \mathbb{D}\to\mathbb{D} $ with $\omega(0)=0$ such that
     \begin{align}
			e^{i\alpha}\left(1+\frac{zf^{\prime\prime}(z)} {f^{\prime}(z)}\right)=\frac{e^{i\alpha}+e^{-i\alpha}\omega(z)}{1-\omega(z)}.
			\end{align}
	Since $\omega(0)=0$, we may set $\omega(z)=z\phi(z)$ for some analytic function $\phi$ that satisfy $\phi(\mathbb{D})\subseteq\mathbb{D}$, \emph{i.e.}, $\phi\in\mathbb{B}$. Therefore, the last equation \eqref{Eq-4.1} reduces to
	\begin{align}\label{Eq-4.3}
		\frac{f^{\prime\prime}(z)}{f^{\prime}(z)}=\frac{2e^{-i\alpha} \cos\alpha\;\omega(z)}{z(1-\omega(z))}
	\end{align}
    which implies that
	\begin{align}\label{Eq-22.88}
	\frac{f^{\prime\prime}(z)}{f^{\prime}(z)}=\frac{2e^{-i\alpha} \cos\alpha\;\phi(z)}{(1-z\phi(z))}.
	\end{align}
    Rewriting \ref{Eq-22.88} yields
	\begin{align}\label{Eq-4.4}
		\phi(z)=\frac{\frac{f^{\prime\prime}(z)}{f^{\prime}(z)}}{2e^{-i\alpha} \cos\alpha+\frac{zf^{\prime\prime}(z)}{f^{\prime}(z)}}\; \mbox{for}\; z\in\mathbb{D}.
	\end{align}
	Since $|\phi(z)|^2\leq 1$,the last equation yields that
	\begin{align}\label{Eq-22.99}
		\bigg|\frac{f^{\prime\prime}(z)}{f^{\prime}(z)}\bigg|^2\leq\left(2e^{-i\alpha} \cos\alpha+\frac{zf^{\prime\prime}(z)}{f^{\prime}(z)}\right)\left(\overline{2e^{-i\alpha} \cos\alpha+\frac{zf^{\prime\prime}(z)}{f^{\prime}(z)}}\right).
	\end{align}
	A simple computation shows that
	\begin{align}
		(1-|z|^2)\bigg|\frac{f^{\prime\prime}(z)}{f^{\prime}(z)}\bigg|^2\leq4\cos^\alpha+4\cos\alpha\;\mathrm{Re}\left(e^{i\alpha}\frac{f^{\prime\prime}(z)}{f^{\prime}(z)}\right).
	\end{align}
	By factorizing, we easily obtain that
	\begin{align*}
		(4\cos\alpha)\left[\cos\alpha+\mathrm{Re}\left(e^{i\alpha}\frac{f^{\prime\prime}(z)}{f^{\prime}(z)}\right)\right]\ge(1-|z|^2)\bigg|\frac{f^{\prime\prime}(z)}{f^{\prime}(z)}\bigg|^2
	\end{align*}
	which implies that
	\begin{align*}
		{\rm Re}\left(e^{i\alpha}\frac{zf^{\prime\prime}(z)}{f^{\prime}(z)}\right)\geq- \cos\alpha+\left(\frac{1-|z|^2}{4\cos\alpha}\right)\bigg|\frac{f^{\prime\prime}(z)}{f^{\prime}(z)}\bigg|^2.
	\end{align*}
     When $\alpha\in(-\pi/2,\pi/2)$, we have
	\begin{align}\label{Eq-2.9}
		{\rm Re}\left(1+e^{i\alpha}\frac{zf^{\prime\prime}(z)}{f^{\prime}(z)}\right)\geq1- \cos\alpha+\left(\frac{1-|z|^2}{4\cos\alpha}\right)\bigg|\frac{f^{\prime\prime}(z)}{f^{\prime}(z)}\bigg|^2.
	\end{align}
    If the equality in \eqref{Eq-2.5} holds for some point $z_0\in\mathbb{D}$, then $|\phi(z_)|=1$ and hence, $\phi(z)\equiv \zeta$ for some $\zeta\in\mathbb{T}$. This gives by \eqref{Eq-22.88} that,
    \begin{align*}
        \frac{f^{\prime\prime}(z)}{f^{\prime}(z)}=\frac{2\zeta e^{-i\alpha}\cos\alpha}{1-z\zeta}\; \mbox{for}\; z\in\mathbb{D},
    \end{align*}
which by integration shows that \begin{align*}
                f^{\prime}(z)=\frac{1}{\left(1-z\zeta \right)^{2e^{-i\alpha}\cos\alpha}}
            \end{align*}
            for some $\zeta\in\mathbb{T}$.\vspace{2mm}

	Next, we will prove that (ii) is equivalent to (iii). \\

    \noindent Multiplying \eqref{Eq-22.99} by $(1-|z|^2)$ both side, we have
	\begin{align*}
		(1-|z|^2)^2\bigg|\frac{f^{\prime\prime}(z)}{f^{\prime}(z)}\bigg|^2\leq4\cos^2\alpha(1-|z|^2)+4\cos\alpha(1-|z|^2){\rm Re}\left(e^{i\alpha}\frac{zf^{\prime\prime}(z)}{f^{\prime}(z)}\right)
	\end{align*}
	which implies that
	\begin{align*}
		&(1-|z|^2)^2\bigg|\frac{f^{\prime\prime}(z)}{f^{\prime}(z)}\bigg|^2-4\cos\alpha(1-|z|^2){\rm Re}\left(e^{i\alpha}\frac{zf^{\prime\prime}(z)}{f^{\prime}(z)}\right)+4\cos^2\alpha\;|z|^2\\&\leq4\cos^2\alpha.
	\end{align*}
	Thus, we have
	\begin{align*}
	\bigg|	(1-|z|^2)\left(\frac{f^{\prime\prime}(z)}{f^{\prime}(z)}\right)-(2\cos\alpha)\bar{z}\bigg|\leq2\cos\alpha.
	\end{align*}
	This completes the proofs.
\end{proof}
We have the following immediate result from Theorem \ref{Th-4.1}.
\begin{corollary}\label{Cor-2.1}
If $f\in \mathcal{S_\alpha}$, then for $\alpha\in(-\pi/2,0]$,  we have
	\begin{align*}
		{\rm Re}\left(1+\frac{zf^{\prime\prime}(z)}{f^{\prime}(z)}\right)\geq 1-\left(\frac{2e^{-i\alpha} \cos\alpha}{2}\right)+\frac{1}{2}\left(\frac{1-|z|^2}{2e^{-i\alpha} \cos\alpha}\right)\bigg|\frac{zf^{\prime\prime}(z)}{f^{\prime}(z)}\bigg|^2.
	\end{align*}
\end{corollary}
\begin{remark}
	In particular, if $\alpha=0$, then we see that $\mathcal{S}_0\subset\mathcal{C}$, the class of all convex functions. In this case, from Corollary \ref{Cor-2.1}, we have the lower bound of ${\rm Re}\left(1+\frac{zf^{\prime\prime}(z)}{f^{\prime}(z)}\right)$ as
	\begin{align*}
		{\rm Re}\left(1+\frac{zf^{\prime\prime}(z)}{f^{\prime}(z)}\right)\geq \frac{1}{4}\left(1-|z|^2\right)\bigg|\frac{f^{\prime\prime}(z)}{f^{\prime}(z)}\bigg|.
	\end{align*}
	This exactly coincides with the bound \cite[Eq. (3)]{Carrasco-Hernandez-AMP-2023} for the class $\mathcal{C}$. Thus, we say that our result, Theorem \ref{Th-4.1}, improves that of \cite{Carrasco-Hernandez-AMP-2023}.
\end{remark}
\begin{remark}
	Also, we have the following result from Theorem \ref{Th-4.1}, which in a particular case, when $\alpha=0$, reduces to \cite[Eq. (4)]{Carrasco-Hernandez-AMP-2023}.
\end{remark}
\begin{remark}
	This relation
	\begin{align*}
		\bigg|	(1-|z|^2)\left(\frac{f^{\prime\prime}(z)}{f^{\prime}(z)}\right)-(2\cos\alpha)\bar{z}\bigg|\leq2\cos\alpha,
	\end{align*}
	have significant role in radius problem, like radius of concavity and convexity.
\end{remark}
Therefore, we propose the following question for further study on radius problem.
\begin{problem}
	Is it possible to establish the radius of concavity and convexity for the function class $\mathcal{S}_{\alpha}$?
\end{problem}
\begin{corollary}\label{Cor-2.2}
	If $f\in \mathcal{S}_0\subset\mathcal{C}$, then we have\begin{align*}
		\bigg|(1-|z|^2)\frac{f^{\prime\prime}(z)}{f^{\prime}(z)}-2\bar{z}\bigg|\leq2.
	\end{align*}
\end{corollary}
In the next result, we establish the distortion theorem, and growth theorem for functions in the class $\mathcal{S}_{\alpha}$.
\begin{theorem}\label{Th-4.2}
	If $f\in \mathcal{S}_{\alpha}$ and $-\pi/2<\alpha<\pi/2$,  then for all $z\in\mathbb{D}$,
	\begin{align*}
		\frac{1}{(1+|z|^2)^{\cos\alpha}}\leq|f^{\prime}(z)|\leq\frac{1}{(1-|z|^2)^{\cos\alpha}}
	\end{align*}
	and
	\begin{align*}
		\int_{0}^{|z|}\frac{1}{(1+\xi^2)^{\cos\alpha}} d|\xi|\leq	|f(z)|\leq\int_{0}^{|z|}\frac{1}{(1-\xi^2)^{\cos\alpha}} d|\xi|.
	\end{align*}
	All the bounds are sharp.
\end{theorem}
\begin{proof}[\bf Proof of Theorem \ref{Th-4.2}]
	Let $f\in\mathcal{S_\alpha}$. From \eqref{Eq-4.4}, we obtain $\phi(0)=0$. Applying the Schwarz Lemma, we then obtain
	\begin{align}
		\bigg|\frac{\frac{f^{\prime\prime}(z)}{f^{\prime}(z)}}{2e^{-i\alpha} \cos\alpha+\frac{zf^{\prime\prime}(z)}{f^{\prime}(z)}}\bigg|^2\leq|z|^2
	\end{align}
	which implies that
	\begin{align*}
		\bigg|\frac{f^{\prime\prime}(z)}{f^{\prime}(z)}\bigg|^2\leq|z|^2\left(2e^{-i\alpha} \cos\alpha+\frac{zf^{\prime\prime}(z)}{f^{\prime}(z)}\right)\left(\overline{2e^{-i\alpha} \cos\alpha+\frac{zf^{\prime\prime}(z)}{f^{\prime}(z)}}\right).
	\end{align*}
	A simple calculation shows that
	\begin{align}\label{Eq-4.13}
		(1-|z|^4)\bigg|\frac{f^{\prime\prime}(z)}{f^{\prime}(z)}\bigg|^2\leq4\cos^2\alpha\;|z|^2+4\cos\alpha\;|z|^2\;{\rm Re}\left(e^{i\alpha}\frac{zf^{\prime\prime}(z)}{f^{\prime}(z)}\right)
	\end{align}
	Multiplying both sides of \eqref{Eq-4.13} by $(1-|z|^4)$, we obtain
	\begin{align*}
		(1-|z|^4)^2&\bigg|\frac{f^{\prime\prime}(z)}{f^{\prime}(z)}\bigg|^2-\cos\alpha\;|z|^2(1-|z|^4){\rm Re}\left(e^{i\alpha}\frac{zf^{\prime\prime}(z)}{f^{\prime}(z)}\right)\\&\leq 4\cos^2\alpha\;|z|^2(1-|z|^4).
	\end{align*}
By adding $\left(2\cos\alpha|z|^3\right)^2$ to both sides of the above inequality, we obtain
	\begin{align}
		(1-|z|^4)^2&\bigg|\frac{f^{\prime\prime}(z)}{f^{\prime}(z)}\bigg|^2-\cos\alpha\;|z|^2(1-|z|^4){\rm Re}\left(e^{i\alpha}\frac{zf^{\prime\prime}(z)}{f^{\prime}(z)}\right)+\left(2\cos\alpha|z|^2\bar{|z|}\right)^2\\&\leq 4\cos^2\alpha\;|z|^2(1-|z|^4)+\left(2\cos\alpha|z|^2\bar{|z|}\right)^2\nonumber.
	\end{align}
	Multiplying both sides by $|z|$, a simple calculation yields
	\begin{align}
		\bigg|(1-|z|^4)\frac{zf^{\prime\prime}(z)}{f^{\prime}(z)}-(2 \cos\alpha)|z|^4\bigg|\leq(2 \cos\alpha)|z|^2
	\end{align}
	which implies
	\begin{align}
		\frac{-(2 \cos\alpha)|z|^2}{1+|z|^2}\leq{\rm Re}\left(\frac{zf^{\prime\prime}(z)}{f^{\prime}(z)}\right)\leq\frac{(2 \cos\alpha)|z|^2}{1-|z|^2}.
	\end{align}
	Substituting $z=re^{i\theta}$, we obtain
	\begin{align*}
		\frac{-(2 \cos\alpha)r}{1+r^2}\leq\frac{\partial }{\partial r}\left(\log|f^{\prime}(re^{i\theta})|\right)\leq\frac{(2 \cos\alpha)r}{1-r^2}.
	\end{align*}
	\noindent{\bf Case A.} When $\alpha=0$, if we integrate with respect to $r$, we obtain
	\begin{align}
		\frac{1}{(1+|z|^2)}\leq|f^{\prime}(z)|\leq\frac{1}{(1-|z|^2)}
	\end{align}
	\noindent{\bf Case B.} When $\alpha\neq0$, if we integrate respect to $r$, we obtain
	\begin{align}
		\frac{1}{(1+|z|^2)^{\cos\alpha}}\leq|f^{\prime}(z)|\leq\frac{1}{(1-|z|^2)^{\cos\alpha}}.
	\end{align}
	Thus the distortion theorem is established.\vspace{1.2mm}
	
	Next, for the growth part of the theorem, from the upper bound it follows that
	\begin{align}
		|f^{\prime}(re^{i\theta})|=\bigg|\int_{0}^{r}f^{\prime}(re^{i\theta})e^{i\theta} dt\bigg|\leq\int_{0}^{r}|f^{\prime}(re^{i\theta})| dt\leq\int_{0}^{r}\frac{1}{(1-t^2)^{\cos\alpha}} dt
	\end{align}
	which implies
	\begin{align}
		|f(z)|\leq\int_{0}^{|z|}\frac{1}{(1-\xi^2)^{\cos\alpha}} d|\xi|
	\end{align}
	for all $z\in\mathbb{D}$. It is well-known that if $f(z_0)$ is a point of minimum modulus on the image of the circle $|z|=r$ and $\gamma=f^{-1}(\Gamma)$, where $\Gamma$ is the line segment from $0$ to $f(z_0)$, then we have 
	\begin{align}
		|f(z)|\geq	|f(z_0)|\geq\int_{0}^{r}\frac{1}{(1+\xi^2)^{\cos\alpha}} d|\xi|.
	\end{align}
	This completes the proof.
\end{proof}
Next, we find the sharp bounds of the pre-Schwarzian and Schwarzian norms for the class $\mathcal{S_\alpha}$ with respect to $f^{\prime\prime}(0)$, assuming $f^{\prime\prime}(0)=0$. The following lemma is key to proving this result.
\begin{lem}\label{lemA}
{\rm \cite{Carrasco-Hernandez-AMP-2023}}
	If $\phi(z):\mathbb{D}\rightarrow\mathbb{D}$ is an analytic function, then
    \begin{align}
		\frac{|\phi(z)|^2}{1-|\phi(z)|^2}\leq\frac{(\phi(0)+|z|)^2}{(1-|\phi(0)|)^2(1-|z|^2)|)}
	\end{align}
\end{lem}
In the next result, we present the sharp bound of the pre-Schwarzian norm for functions in the class $\mathcal{S_\alpha}$.
\begin{theorem}\label{Th-4.3}
	If $f\in\mathcal{S_\alpha}$, then for all $z\in\mathbb{D}$ and $-\pi/2<\alpha<\pi/2$, then
	\begin{align}\label{Eq-2.24}
		||Pf||\leq2\cos\alpha.
	\end{align}
	The inequality \eqref{Eq-2.24} is sharp.
\end{theorem}
\begin{proof}[\bf Proof of Theorem \ref{Th-4.3}]
	Since $\phi(z)=z\xi(z)$, with $|\xi(z)|<1$, then in \eqref{Eq-2.4} we obtain
	\begin{align*}
		\sup_{z\in\mathbb{D}}(1-|z|^2)\bigg|\frac{f^{\prime\prime}(z)}{f^{\prime}(z)}\bigg|&\leq\sup_{z\in\mathbb{D}}(1-|z|^2)\frac{|2e^{-i\alpha} \cos\alpha \;z\xi(z)|}{1-|z|^2|\xi(z)|}\\&\leq(2 \cos\alpha) \sup_{0\leq r\leq 1}\frac{r(1-r^2)}{(1-r^2)}\\&=2 \cos\alpha.
	\end{align*}
	Thus, the desired inequality \eqref{Eq-2.24} is obtained.\vspace{1.2mm}
	
	The next part of the proof is to show that the inequality \eqref{Eq-2.24} is sharp. To this end, we consider the extremal function given by
	\begin{align}
		f^*(z)=	\int_{0}^{z}\frac{1}{(1-\xi^2)^{\cos\alpha}} d\xi.
	\end{align}
	It can be easily shown that $||Pf^*||=2 \cos\alpha.$	
\end{proof}
We have the following immediate result from Theorem \ref{Th-4.3}.
\begin{corollary}
	If $f\in\mathcal{S_\alpha}$, then for all $z\in\mathbb{D}$ and $\alpha=0$, we have
	\begin{align}\label{Eq-2.25}
		||Pf||\leq 2.
	\end{align}
	The inequality \eqref{Eq-2.25} is sharp.
\end{corollary}
\subsection*{Sharpness of Corolary 4.3 :}
For $\alpha=0$, it follows from  that

\begin{align*}
	\frac{f^{\prime\prime}_0}{f_0^{\prime}}(z)=\frac{2z}{1-z^2}\;\;\mbox{and}\;\;Pf_0=\frac{2}{1-z^2}.
\end{align*}
A simple computation thus yields that
\begin{align*}
	||Pf_0||=\sup_{z\in\mathbb{D}}\left(1-z^2\right)|Pf_0|=\sup_{z\in\mathbb{D}}\left(1-|z|^2\right)\frac{2}{1-|z|^2}=2
\end{align*}
and we see the constant $2$ is sharp.
\vspace*{2mm}

Our next result provides a sharp bound for the Schwarzian norm when $f\in \mathcal{S_\alpha}$. The method adopted for the proof is a direct application of the Schwarz Lemma.
\begin{theorem}\label{Th-4.4}
		If $f\in\mathcal{S_\alpha}$, for all $z\in\mathbb{D}$ and $-\pi/2<\alpha<\pi/2$, then the Schwarzian norm
\begin{align}\label{Eq-2.28}
			\|Sf\|=\sup\limits_{z\in \mathbb{D}}(1-|z|^2)^2|Sf(z)|\leq 2\cos\alpha\left(2-\cos\alpha\right).
\end{align}
		The inequality \eqref{Eq-2.28} is sharp.
\end{theorem}
For the classical class $\mathcal{S}_0\subset\mathcal{C}$ of convex functions, we have the following result from Theorem \ref{Th-4.4}.
\begin{corollary}
		If $f\in\mathcal{S}_0\subset\mathcal{C}$, then for all $z\in\mathbb{D}$, we have
		\begin{align}\label{Eq-2.29}
			\|Sf\|\leq 2.
		\end{align}
		The inequality \eqref{Eq-2.29} is sharp.
\end{corollary}
	\begin{proof}[\bf Proof of Theorem \ref{Th-4.4}]
	From \eqref{Eq-2.4}, we have
	\begin{align*}
		\frac{f^{\prime\prime}(z)}{f^{\prime}(z)}=\frac{(2e^{-i\alpha} \cos\alpha)\phi(z)}{(1-z\phi(z))}.
	\end{align*}
	A straightforward computation shows that
	\begin{align}
		Sf(z)=(2e^{-i\alpha} \cos\alpha)\left[\frac{\phi^{\prime}(z)+\left(1-e^{-i\alpha} \cos\alpha\right)\phi^2(z)}{(1-z\phi(z))^2}\right].
	\end{align}
	By using triangle inequality and Schwarz pick lemma, we obtain
	\begin{align}
		(1-|z|^2)^2|Sf|&\leq(2 \cos\alpha)\bigg|\phi^{\prime}(z)+\left(1-e^{-i\alpha} \cos\alpha\right)\phi^2(z)\bigg|\frac{(1-|z|^2)^2}{|1-z\phi(z)|^2}\\&=\frac{(2\cos\alpha)(1-|z|^2)^2}{|1-z\phi(z)|^2}\left(\frac{1-|\phi(z)|^2}{1-|z|^2}+\left(1-|e^{i\alpha}\cos\alpha|\right)|\phi(z)|^2\right).
	\end{align}\label{Eq-4.29}
	We define the function $\Psi(z):\mathbb{D}\rightarrow\mathbb{D}$ such that
	\begin{align}
		\Psi(z):=\frac{\bar{z}-\phi(z)}{1-z\phi(z)}.
	\end{align}
	Since $\phi(\mathbb{D})\subseteq\mathbb{D}$ then $(1-|z|^2)(1-|z\phi(z)|^2)>0$, it follows that
	\begin{align}
		|\bar{z}-\phi(z)|^2<|1-z\phi(z)|^2.
	\end{align}
	Thus, we conclude  that $|\Psi(z)|^2<1$. Therefore,
	\begin{align}\label{Eq-2.322}
		1-|\Psi(z)|^2=\frac{(1-|\phi(z)|^2)(1-|z|^2)}{|1-z\phi(z)|^2}
	\end{align}
	and
	\begin{align}\label{Eq-2.333}
		\frac{(1-|z|^2)^2}{|1-z\phi(z)|^2}=\frac{(1-|\Psi(z)|^2)(1-|z|^2)}{(1-|\phi(z)|^2)}.
	\end{align}
	If we substitute the expression \eqref{Eq-2.333} and \eqref{Eq-2.322} we have
	\begin{align}
		(1-|z|^2)^2|S_f(z)|\leq(2\cos\alpha)(1-|\Psi_1(z)|^2)\left(1+\left(1-|e^{i\alpha}\cos\alpha|\right)\frac{|\phi(z)|^2(1-|z|^2)}{(1-|\phi(z)|^2)}\right).
	\end{align}
	Since $f^{\prime\prime}(0)=0$ implies that $\phi(0)=0$, using Lemma A, we obtain
	\begin{align}\label{Eq-4.35}
		\frac{|\phi(z)|^2}{1-|\phi(z)|^2}\leq\frac{|z|^2}{1-|z|^2}.
	\end{align}
	Using \eqref{Eq-4.35} in \eqref{Eq-4.29}, we obtain
	\begin{align*}
		(1-|z|^2)^2|S_f(z)|\leq(2\cos\alpha)(1-|\Psi(z)|^2)\left(1+\left(1-\cos\alpha\right)|z|^2\right).
	\end{align*}
	Again, since $1-|\Psi(z)|^2 \leq 1$, we see that
	\begin{align*}
		(1-|z|^2)^2|S_f(z)|&\leq(2\cos\alpha)\left(1+(1-\cos\alpha)\right)\\&=2\cos\alpha\left(2-\cos\alpha\right).\nonumber
	\end{align*}
        Thus, the inequality is established.	The sharpness follows from the Example \ref{Example-4.1}.
	\end{proof}
\begin{example}\label{Example-4.1}
	The family of parameterized functions defined as:
	\begin{align}
		f_\alpha (z)=\int_{0}^{z}\frac{1}{(1-\xi^2)^{\cos\alpha}} d\xi,\;\;\;\;\mbox{for}\;\;-\pi/2<\alpha<\pi/2
	\end{align}
	maximizes the Schwarzian norm defined as:
	\begin{align*}
		||S_f||=\sup_{z\in\mathbb{D}}(1-|z|^2)^2|Sf|
	\end{align*}
	and from this, the sharpness of the inequality holds for $-\pi/2<\alpha<\pi/2$. Note that
	\begin{align}
		\frac{f^{\prime\prime}_\alpha}{f_\alpha^{\prime}}(z)=\frac{(2 \cos\alpha)z}{1-z^2}\;\;\mbox{and}\;\;Sf_\alpha=\frac{(2 \cos\alpha)}{(1-z^2)^2}\left[1+\left(1- \cos\alpha\right)z^2\right]
	\end{align}
	which implies that
	\begin{align}
		||S_{f_{\alpha}}||&=\sup_{z\in\mathbb{D}}(1-|z|^2)^2|Sf_\alpha|\\&=2 \cos\alpha\left(2-\cos\alpha\right).	\end{align}
	In general, the integral formula for $f_\alpha$ given in above does not give primitives in terms of elementary functions, however when $\alpha=0$
	\begin{align*}
		f_0 (z)=\int_{0}^{z}\frac{1}{(1-\xi^2)} d\xi=\frac{1}{2}\log\left(\frac{1+z}{1-z}\right),
	\end{align*}
	where $||Sf_0||=2$.
\end{example}

For functions $f$ in the class $\mathcal{S_\alpha}$, we obtain the following result which finding an upper bound of $(1-|z|^2)^2|Sf(z)|$.
\begin{theorem}\label{Th-4.5}
	If $f\in\mathcal{S_\alpha}$, for all $z\in\mathbb{D}$ and $-\pi/2<\alpha<\pi/2$ and
	\begin{align}
		\gamma=|\phi(0)|=\frac{|f^{\prime\prime}(0)|}{2\cos\alpha},
	\end{align}
	then
	\begin{align*}
		(1-|z|^2)^2|S_f(z)|\leq(2\cos\alpha)\left(1+(1-\cos\alpha)\frac{1+\gamma}{1-\gamma}\right).
	\end{align*}
\end{theorem}
\begin{corollary}
	If $f\in\mathcal{S}_0\subset\mathcal{C}$, for all $z\in\mathbb{D}$ with
	\begin{align}
		\gamma=|\phi(0)|=\frac{|f^{\prime\prime}(0)|}{2\cos\alpha},
	\end{align}
	then
	\begin{align*}
		(1-|z|^2)^2|S_f(z)|\leq(2\cos\alpha)\left(1+(1-\cos\alpha)\frac{1+\gamma}{1-\gamma}\right).
	\end{align*}
\end{corollary}
\begin{proof}[\bf Proof of Theorem \ref{Th-4.5}]
	Let $\gamma=|\phi(0)|$. Applying the Lemma A, we  calculate
	\begin{align}
		\frac{|\phi(z)|^2}{1-|\phi(z)|^2}\leq\frac{(\gamma+|z|)^2}{(1-\gamma^2)(1-|z|^2)}.
	\end{align}
	If we substitute in (2.30), we get
	\begin{align}
		(1-|z|^2)^2|S_f(z)|\leq(2\cos\alpha)(1-|\Phi_1(z)|^2)\left(1+(1-\cos\alpha)\frac{(\gamma+|z|)^2}{(1-\gamma^2)}\right).
	\end{align}
	From the fact that $|z|<1$ and  $1-|\Phi_1(z)|^2\leq1$, we can easily calculate
	\begin{align}
		(1-|z|^2)^2|S_f(z)|\leq(2\cos\alpha)\left(1+(1-\cos\alpha)\frac{1+\gamma}{1-\gamma}\right).
	\end{align}
	This completes the proof.
\end{proof}

\vskip .20in

\noindent\textbf{Acknowledgments.} The authors would like to the thank the referee for their helpful suggestions and comments to improve the the paper.\vspace{1.5mm}

\noindent\textbf{Compliance of Ethical Standards.}\vspace{2mm}

\noindent\textbf{Conflict of interest.} The authors declare that there is no conflict of interest regarding the publication of this paper.\vspace{1.5mm}

\noindent\textbf{Data availability statement.} Data sharing not applicable to this article as no datasets were generated or analyzed during the current study.\vspace{1.5mm}

\noindent\textbf{Funding} No funds.

\end{document}